\documentclass[12pt, reqno]{amsart}
\usepackage{amsmath, amsthm, amscd, amsfonts, amssymb, graphicx, color}
\usepackage[bookmarksnumbered, colorlinks, plainpages]{hyperref}
\hypersetup{colorlinks=true,linkcolor=red, anchorcolor=green, citecolor=cyan, urlcolor=red, filecolor=magenta, pdftoolbar=true}
\input{mathrsfs.sty}
\newtheorem{theorem}{Theorem}[section]
\newtheorem{lemma}[theorem]{Lemma}
\newtheorem{proposition}[theorem]{Proposition}
\newtheorem{corollary}[theorem]{Corollary}
\theoremstyle{definition}

\newtheorem{example}[theorem]{Example}

\theoremstyle{remark}

\numberwithin{equation}{section}

\begin{document}

\title[Characterizations of smooth spaces by $\rho_*$-orthogonality]{Characterizations of smooth spaces by $\rho_*$-orthogonality}
\author[M.S. Moslehian, A. Zamani, M. Dehghani]
{Mohammad Sal Moslehian, Ali Zamani, \MakeLowercase{and} Mahdi Dehghani}
\address [Moslehian]{Department of Pure Mathematics, Ferdowsi University of Mashhad, P.O. Box 1159, Mashhad 91775, Iran} \email{moslehian@um.ac.ir, moslehian@member.ams.org}
\address [Zamani]{Department of Mathematics, Farhangian University, Semnan, Iran} \email{zamani.ali85@yahoo.com}
\address [Dehghani]{Department of Pure Mathematics, Faculty of Mathematical Sciences, University of Kashan, P. O. Box 87317-53153, Kashan, Iran} \email{m.dehghani@kashanu.ac.ir, e.g.mahdi@gmail.com}



\begin{abstract}
The aim of this paper is to present some results concerning the $\rho_*$-orthogonality in real normed spaces and its
preservation by linear operators. Among other things, we prove that if $T\,: X \longrightarrow Y$ is a nonzero linear $(I, \rho_*)$-orthogonality preserving mapping between real normed spaces, then
$$\frac{1}{3}\|T\|\|x\|\leq\|Tx\|\leq 3[T]\|x\|, \qquad (x\in X)$$
where $[T]:=\inf\{\|Tx\|: \,x\in X, \|x\|=1\}$. We also show that the pair $(X,\perp_{\rho_*})$ is an orthogonality space in the sense of R\"{a}tz. Some characterizations
of smooth spaces are given based on the $\rho_*$-orthogonality.
\end{abstract}
 
\maketitle

\section{Introduction}
Throughout the paper, $(X,\|\cdot\|)$ denote a real normed space of dimension at least $2$. If the norm of $X$ comes from an inner product $\langle\cdot,\cdot\rangle$, then there is a natural orthogonality relation defined by
\begin{align*}
x\perp y\Leftrightarrow\langle x,y\rangle=0 \qquad(x,y\in X).
\end{align*}
A mapping $[\cdot|\cdot]:X\times X\rightarrow\mathbb{R}$ satisfying
\begin{itemize}
\item[(i)] $[r x+ ty|z]=r[x|z]+t[y|z]$;
\item[(ii)] $[x|x]=\|x\|^2$;
\item[(iii)] $|[x|y]|\leq\|x\|\|y\|$,
\end{itemize}
for all $x,y,z\in X$ and all $r, t\in\mathbb{R}$ is called a semi-inner product in $X$. There can be infinitely many such semi-inner products. It is well known that in a normed space $X$, there exists exactly one semi-inner product if and only if $X$ is smooth (i.e., there is a unique supporting hyperplane at each point of the unit sphere $X$, or equivalently, the norm is G\^{a}teaux differentiable on $X$); see \cite{Dra}.

For a given semi-inner product and vectors $x,y\in X$, the semi-inner product orthogonality is defined as follows:
\begin{align*}
x\perp_s y\Leftrightarrow [y|x]=0.
\end{align*}

There are several concepts of orthogonality such as Birkhoff--James and isosceles in an arbitrary normed space $X$, which are recalled as follows (see \cite{A.S.T} and references therein):
\begin{itemize}
\item[(i)] The Birkhoff--James orthogonality $\perp_B$: $x\perp_B y$ if and only if $\|x\|\leq\|x+\lambda y\|$ for all $\lambda\in\mathbb{R}$.
\item[(ii)] The isosceles orthogonality $\perp_I$: $x\perp_I y$ if and only if $\|x+y\|=\|x-y\|$.
\end{itemize}
The following mapping $\langle\cdot,\cdot\rangle_g:X\times X\rightarrow\mathbb{R}$ was introduced by Mili\v{c}i\'{c} \cite{Mil} as follows:
\begin{align*}
\langle y,x\rangle_g=\frac{\rho_-(x,y)+\rho_+(x,y)}{2},
\end{align*}
where the mappings $\rho_-,\rho_+:X\times X\rightarrow\mathbb{R}$ are defined by
\begin{align*}
\rho_{\pm}(x,y):=\lim_{t\rightarrow0^{\pm}}\frac{\|x+ty\|^2-\|x\|^2}{2t}=\|x\|\lim_{t\rightarrow0^{\pm}}\frac{\|x+ty\|-\|x\|}{t}
\end{align*}
and called norm derivatives. In addition, some orthogonality relations are introduced by
\begin{align*}
x\perp_{\rho_-}y\quad\mbox{if and only if}\quad \rho_-(x,y)=0,\\
x\perp_{\rho_+}y\quad\mbox{if and only if}\quad\rho_+(x,y)=0
\end{align*}
and
\begin{align*}
x\perp_{\rho}y\quad\mbox{if and only if}\quad \rho(x,y):=\langle y,x\rangle_g=0.
\end{align*}
For more information about the norm derivatives and their fundamental properties the reader is referred to \cite{A.S.T, C.W.2, C.W.3}.
In \cite{C.L}, the authors introduced the notion of $\rho_*$-orthogonality. Let $x,y\in X$. Then $x$ is $\rho_*$-orthogonal to $y$ (denoted by $x\perp_{\rho_*}y$) if $$\rho_*(x,y) := \rho_-(x,y)\rho_+(x,y) = 0.$$
The main aim of the present work is to investigate the notion of $\rho_*$-orthogonality on a normed space $X$. It is clear that $\perp_{\rho_-}\cup\perp_{\rho_+} = \perp_{\rho_*}\,\subseteq \,\perp_{B}$. But the following examples show that for non-smooth spaces there may be not any one of $\perp_{B}\subseteq\perp_{\rho_*}$, $\perp_{\rho}\subseteq\perp_{\rho_*}$, $\perp_{\rho_*}\subseteq\perp_{\rho}$, $\perp_{\rho_*}\subseteq\perp_{\rho_+}$ and $\perp_{\rho_*}\subseteq\perp_{\rho_-}$ holds.

\begin{example}\label{Ex1}
Consider the space $X=\mathbb{R}^3$ equipped with the norm $\|(x_1,x_2,x_3)\|=|x_1|+|x_2|+|x_3|$. If $x=(1,0,0)$, $y=(1,1,1)$, $z=(1,1,0)$ and $w=(-1,1,0)$, then
\begin{align*}
\|x\|=1\leq |1+\lambda|+2|\lambda|=\|x+\lambda y\| \qquad(\lambda\in\mathbb{R}).
\end{align*}
Hence $x\perp_{B}y$. On the other hand, we have
\begin{align*}
\rho_*(x,y) =\lim_{t\rightarrow 0^-}\frac{|1+t|+2|t|-1}{t}\times\lim_{t\rightarrow 0^+}\frac{|1+t|+2|t|-1}{t}= -1\times 3 = -3.
\end{align*}
Thus $x$ is not $\rho_*$-orthogonal to $y$.
Further, we have
\begin{align*}
\rho_*(x,z) =\lim_{t\rightarrow 0^-}\frac{|1+t|+|t|-1}{t}\times\lim_{t\rightarrow 0^+}\frac{|1+t|+|t|-1}{t} = 0\times2= 0.
\end{align*}
Therefore, $x\perp_{\rho_*}z$ but neither $x\perp_{\rho_+}z$ nor $x\perp_{\rho}z$. Similarly, we have $x\perp_{\rho_*}w$ but not $x\perp_{\rho_-}w$.
\end{example}
\begin{example}\label{Ex2}
Let $X=\mathbb{R}^2$ endowed with the norm $\|(x_1,x_2)\|=\max\{|x_1|, |x_2|\}$. If $x=(1,0)$ and $y=(1,-1)$, then it is easy to see that
\begin{align*}
\rho_*(x,y) = \rho_-(x,y)\rho_+(x,y) = -1 \times 1 = -1.
\end{align*}
Hence $x\perp_{\rho}y$ but not $x\perp_{\rho_*}y$.
\end{example}

A mapping $T:H\rightarrow K$ between two inner product spaces $H$ and $K$ is said to be orthogonality preserving if $x\perp y$ ensures $Tx\perp Ty$ for every $x,y\in H$. It is well known that an orthogonality preserving linear mapping between two inner product spaces is necessarily a similarity, i.e., a scalar multiple of an isometry; see \cite{Ch.0, Z.M.F, Z.C.H.K}.
Now, let $X$ and $Y$ be normed spaces and let $\diamondsuit \in \{B, I, s, \rho_-, \rho_+, \rho, \rho_*\}$. Let us consider linear mappings $T:X\rightarrow Y$ which preserve the $\perp_{\diamondsuit}$ orthogonality in the following sense:
$$x\perp_{\diamondsuit} y\Rightarrow Tx\perp_{\diamondsuit} Ty \qquad(x,y\in X).$$

For $\diamondsuit \in \{B, s\}$, it has been proved by Blanco and Turn\v{s}ek \cite{B.T} that a linear mapping preserving $\diamondsuit$-orthogonality has to be a similarity. Martini and Wu \cite{M.W} proved that a linear mapping $T$
preserves $I$-orthogonality if and only if $T$ is a similarity; see also \cite{CW1}. In \cite{C.W.2,C.W.3, W}, for $\diamondsuit \in \{\rho_-, \rho_+, \rho\}$, Chmieli\'{n}ski and W\'{o}jcik proved that a linear mapping which preserves $\diamondsuit$-orthogonality is a similarity. Various kinds of orthogonality preserving mappings have been studied by the authors of the present paper; cf. \cite{Z, Z.M}.

Let us now suppose that $\perp$ is a binary relation on a vector space $X$ with the following properties:

(O1) Totality of $\perp$ for zero: $x \perp 0$ and $0 \perp x$ for all $x\in X$;

(O2) Independence: if $x, y \in X\setminus\{0\}$ and $x \perp y$, then $x$ and $y$ are linearly independent;

(O3) Homogeneity: if $x, y \in X$ and $x \perp y$, then $\alpha x \perp \beta y$ for all $\alpha, \beta \in \mathbb{R}$;

(O4) The Thalesian property: let $P$ be a two-dimensional subspace of $X$. If $x\in P$ and $\lambda \geq 0$, then there exists $y \in P$ such that $x \perp y$ and $x + y\perp \lambda x - y$.

The pair $(X, \perp)$ is called an orthogonality space in the sense of R\"{a}tz \cite{R}. Some examples of special interest are:

(i) The trivial orthogonality on a vector space $X$ defined by (O1), and for nonzero elements $x, y \in X$, $x \perp y$ if and only if $x$ and $y$ are linearly independent.

(ii) The ordinary orthogonality on an inner product space $X$.

(iii) The Birkhoff--James orthogonality on a normed space $X$ (see \cite{A.S.T}).

(iv) The $\rho$-orthogonality on a normed space $X$ (see \cite{A.S.T}).

In Section 2, we first give basic properties of the $\rho_*$-orthogonality. Then we consider classes of linear mappings preserving this kind of orthogonality. In Section 3, we give some characterizations of smooth spaces in terms of $\rho_*$-orthogonal preserving mappings. In the last section we will prove that the pair $(X,\perp_{\rho_*})$ is an orthogonality space. Some other related results are also presented.


\section{$\rho_*$-orthogonality preserving mappings}
We start this section with some properties of the $\rho_*$-orthogonality.
\begin{proposition}\label{pr1}
Let $(X,\|\cdot\|)$ be a normed space. Then
\begin{itemize}
\item[(i)] $\rho_*(tx, y) = \rho_*(x, ty) = t^2\rho_*(x, y)$ for all $x,y\in X$ and all $t\in\mathbb{R}$.
\item[(ii)] $|\rho_*(x,y)|\leq\|x\|^2\|y\|^2$ for all $x,y\in X$.
\item[(iii)] For all nonzero vectors $x,y\in X$, if $x\perp_{\rho_*}y$, then $x$ and $y$ are linearly independent.
\item[(iv)] $\rho_*(x, tx + y) = t^2\|x\|^4 + 2t\|x\|^2\rho(x,y) + \rho_*(x,y)$ for all $x,y\in X$ and all $t\in\mathbb{R}$.
\end{itemize}
\end{proposition}
\begin{proof}
The statements (i) and (ii) follow directly from the definition of $\rho_*$. To establish (iii) suppose that $x,y\in X$ are nonzero elements of $X$ and $x\perp_{\rho_*}y$. Assume that there exists a nonzero $t\in\mathbb{R}$ such that $x=ty$. Then
\begin{align*}
\rho_*(x,y)=\rho_*(ty,y)=t^2\|y\|^4=0.
\end{align*}
It follows that $t=0$, which is impossible. Therefore, $x,y$ are linearly independent.

To prove (iv), assume that $x,y\in X$ and $t\in\mathbb{R}$. By the basic properties of $\rho_{\pm}$, we have $\rho_{\pm}(x, tx+y)=t\|x\|^2+\rho_{\pm}(x,y)$ \cite[Theorem 2.1.1]{A.S.T}. Therefore,
\begin{align*}
\rho_*(x, tx+y)&=\rho_-(x, tx+y)\rho_+(x, tx+y)\\
&=\Big(t\|x\|^2+\rho_-(x,y)\Big)\Big( t\|x\|^2+\rho_+(x,y)\Big)\\
&=t^2\|x\|^4+t\|x\|^2\Big(\rho_-(x,y)+\rho_+(x,y)\Big)+\rho_-(x,y)\rho_+(x,y)\\
&=t^2\|x\|^4+2t\|x\|^2\rho(x,y)+\rho_*(x,y).
\end{align*}
\end{proof}
\begin{proposition}
Let $(X,\|\cdot\|)$ be a normed space and let $[\cdot|\cdot]$ be a given semi-inner product in $X$. Then the following conditions are equivalent:
\begin{itemize}
\item[(i)] $\perp_{\rho_*}=\perp_{s}$.
\item[(ii)] $\perp_{\rho_*} \subseteq \perp_{s}$.
\item[(iii)] $\rho_*(x,y)=[y|x]^2$ for all $x,y\in X$.
\end{itemize}
\end{proposition}
\begin{proof}
The implications (i)$\Rightarrow$(ii) and (iii)$\Rightarrow$(i) are clear. Next, suppose that (ii) holds and let $x,y\in X$. We have
\begin{align*}
x\perp_{\rho_*}\frac{-\rho_+(x,y)}{\|x\|^2}x+y \qquad \mbox{and} \qquad x\perp_{\rho_*}\frac{-\rho_-(x,y)}{\|x\|^2}x+y.
\end{align*}
It follows from (ii) that
\begin{align*}
\left[\frac{-\rho_+(x,y)}{\|x\|^2}x+y|x\right] = 0 \qquad \mbox{and} \qquad \left[\frac{-\rho_-(x,y)}{\|x\|^2}x+y|x\right] = 0.
\end{align*}
Thus
\begin{align*}
\rho_+(x,y)=[y|x] \qquad \mbox{and} \qquad \rho_-(x,y)=[y|x].
\end{align*}
Therefore
\begin{align*}
\rho_*(x,y)= \rho_-(x,y) \rho_+(x,y) = [y|x]^2.
 \end{align*}
\end{proof}
\begin{proposition}\label{pr.1.1}
Let $X$ be a normed space endowed with two norms ${\|\cdot\|}_1$ and ${\|\cdot\|}_2$. Then the following conditions are equivalent:
\begin{itemize}
\item[(i)] The norms ${\|\cdot\|}_1$ and ${\|\cdot\|}_2$ are equivalent.
\item[(ii)] There exists a positive constant $\alpha$ such that
$$\Big|\rho_{*, 1}(x, y) - \rho_{*, 2}(x, y)\Big| \leq \alpha \min\Big\{\|x\|^2_1\,\|y\|^2_1, \|x\|^2_2\,\|y\|^2_2\Big\} \qquad(x, y \in X).$$
\end{itemize}
\end{proposition}
\begin{proof}
First, we prove (i)$\Rightarrow$(ii). If (i) holds, then there are positive scalars $m, M$ such that
\begin{align}\label{Eq3}
m{\|x\|}_1\leq{\|x\|}_2\leq M{\|x\|}_1 \qquad(x\in X).
\end{align}
On the other hand, Proposition \ref{pr1} (ii) implies that
\begin{align*}
|\rho_{*,j}(x,y)|\leq\|x\|^2_j\|y\|^2_j,\qquad\mbox{($j=1,2$\, and \,$x,y\in X$)}.
\end{align*}
From \eqref{Eq3} we conclude that
\begin{align}\label{Eq4}
\rho_{*,1}(x,y)-\rho_{*,2}(x,y)\leq\|x\|^2_1\|y\|^2_1+\|x\|^2_2\|y\|^2_2\leq(1+M^2)\|x\|^2_1\|y\|^2_1.
\end{align}
Similarly,
\begin{align}\label{Eq5}
\rho_{*,1}(x,y)-\rho_{*,2}(x,y)\leq(1+\frac{1}{m^2})\|x\|^2_2\|y\|^2_2.
\end{align}
It follows from \eqref{Eq4} and \eqref{Eq5} that
\begin{align*}
\Big|\rho_{*,1}(x,y)-\rho_{*,2}(x,y)\Big|\leq\max\Big\{1+M^2, 1+\frac{1}{m^2}\Big\}\min\Big\{\|x\|^2_1\,\|y\|^2_1, \|x\|^2_2\,\|y\|^2_2\Big\}.
\end{align*}
Putting $\alpha=\max\Big\{1+M^2, 1+\frac{1}{m^2}\Big\}$, we reach (ii).

Now, we prove that (ii) yields (i). Let $y=x\in X$. Since $\rho_{*,1}(x,x)=\|x\|^2_1$ and $\rho_{*,2}(x,x)=\|x\|^2_2$, we have
\begin{align*}
\Big|\,\|x\|^4_1-\|x\|^4_2\,\Big|\leq\alpha\|x\|^4_1\quad\mbox{and}\quad \Big|\,\|x\|^4_1-\|x\|^4_2\,\Big|\leq\alpha\|x\|^4_2.
\end{align*}
Hence
\begin{align*}
{\|x\|}_2\leq\sqrt[4]{1+\alpha}{\|x\|}_1\quad\mbox{and}\quad{\|x\|}_1\leq\sqrt[4]{1+\alpha}{\|x\|}_2.
\end{align*}
Put $m=\frac{1}{\sqrt[4]{1+\alpha}}$ and $M=\sqrt[4]{1+\alpha}$ to get
\begin{align*}
m{\|x\|}_1\leq{\|x\|}_2\leq M{\|x\|}_1 \qquad(x\in X).
\end{align*}
Thus the norms ${\|\cdot\|}_1$ and ${\|\cdot\|}_2$ are equivalent.
\end{proof}
Recall that a normed space $(X, \|\cdot\|)$ satisfies the $\delta$-parallelogram law for some $\delta\in[0 , 1)$, if the double inequality
\begin{align}\label{id.0011}
2(1 - \delta)\|z\|^2 \leq \|z + w\|^2 + \|z - w\|^2 - 2\|w\|^2 \leq 2(1 + \delta)\|z\|^2
\end{align}
holds for all $z, w \in X$; cf. \cite{Ch.1}. To get the next result we use some ideas of \cite{Ch.1}.
\begin{proposition}\label{pr.123456}
Suppose that $(X, \|\cdot\|)$ is a normed space such that $\|\cdot\|$ satisfies the $\delta$-parallelogram law for some $\delta\in[0 , 1)$.
Then there exists an inner product $\langle \cdot, \cdot\rangle$ in $X$ such that
$$\Big|\rho_*(x, y) - {\langle x, y\rangle}^2\Big| \leq \frac{2 - \delta}{1 - \delta}\min\Big\{\|x\|^2\,\|y\|^2, |||x|||^2\,|||y|||^2\Big\} \qquad(x, y \in X),$$
where $|||\cdot|||$ is the norm generated by $\langle\cdot, \cdot\rangle$.
\end{proposition}
\begin{proof}
From the $\delta$-parallelogram law (\ref{id.0011}), we get
$$2(1 - \delta)\|z\|^2 \leq \|z + w\|^2 + \|z - w\|^2 - 2\|w\|^2 \leq 2(1 + \delta)\|z\|^2 \qquad(z ,w \in X).$$
Let us define $h(z) := (1 - \delta)\|z\|^2$, $f(z) := \|z\|^2$ and $g(z) := (1 + \delta)\|z\|^2$ for $z\in X$. Then
$$2h(z) \leq f(z + w) + f(z - w) - 2f(w) \leq 2g(z) \qquad(z, w\in X).$$
By \cite[Proposition 3.2]{Ch.1}, there exists a real quadratic mapping $Q\,:X\longrightarrow \mathbb{R}$, i.e., a mapping satisfying
$Q(z + w) + Q(z - w) = 2Q(z) + 2Q(w)\,\,(z, w \in X)$, such that
\begin{align}\label{id.001}
(1 - \delta)\|z\|^2 \leq Q(z) \leq (1 + \delta)\|z\|^2,\qquad(z \in X).
\end{align}
Hence $Q(z) > 0$ for $z \in X\setminus\{0\}$ and $Q(0) = 0$.
Now, define $${\langle z, w\rangle} : = \frac{1}{4}[Q(z + w) - Q(z - w)],\qquad(z ,w \in X).$$
It follows from (\ref{id.001}) that ${\langle \cdot, \cdot\rangle}$ is locally bounded with respect to each variable. Due to $Q$ is quadratic, ${\langle \cdot, \cdot\rangle}$ is symmetric and biadditive. Thus
${\langle \cdot, \cdot\rangle}$ is linear in each variable. Hence ${\langle \cdot, \cdot\rangle}$ is an inner product in $X$ generating the norm $|||z||| := \sqrt{Q(z)}$, $z \in X$. Now, we can write (\ref{id.001}) as
$$\sqrt{1 - \delta}\|z\| \leq |||z||| \leq \sqrt{1 + \delta}\|z\| \qquad(z \in X).$$
Therefore, the norms $\|\cdot\|$ and $|||\cdot|||$ are equivalent. Since $\max\Big\{1+(\sqrt{1 + \delta})^2, 1+\frac{1}{(\sqrt{1 - \delta})^2}\Big\} = \frac{2 - \delta}{1 - \delta}$,
the assertion follows from Proposition \ref{pr.1.1}.
\end{proof}
The next result plays an essential role in our investigation. We should notify that the equivalence (i)$\Leftrightarrow$(iii) of the following result was already
shown in \cite[Theorem 2.2]{C.L}. We shall prove it by a different approach.
\begin{theorem}\label{L1}
Let $X,Y$ be normed spaces and let $T\,:X\longrightarrow Y$ be a nonzero linear mapping.
Then the following conditions are equivalent:
\begin{itemize}
\item[(i)] $x\perp_{\rho_*}y\, \Longrightarrow Tx\perp_{\rho_*}Ty \quad (x, y \in X)$.
\item[(ii)] $x\perp_{\rho_*}y\, \Longrightarrow Tx\perp_B Ty \quad (x, y \in X)$.
\item[(iii)] $\|Tx\| = \|T\|\,\|x\| \quad (x\in X)$.
\item[(iv)] $\rho_*(Tx, Ty) = \|T\|^4\,\rho_*(x, y) \quad (x, y\in X)$.\\
If $X = Y$, then each one of these assertions is also equivalent to
\item[(v)] there exists a semi-inner product $[\cdot|\cdot]\,:X\times X\longrightarrow \mathbb{R}$ satisfying
$$[Tx|Ty] = \|T\|^2[x|y] \qquad (x, y\in X).$$
\end{itemize}
\end{theorem}
\begin{proof}
The implications (i)$\Rightarrow$(ii), (iii)$\Rightarrow$(iv),  (iv)$\Rightarrow$(i) and (v)$\Rightarrow$(iii) are clear.

To prove (ii)$\Rightarrow$(iii), suppose that (ii) holds. Fix $x, y\in X\setminus\{0\}$. If $x$ and $y$ are linearly dependent, then $\frac{\|Tx\|}{\|x\|} = \frac{\|Ty\|}{\|y\|}$. Assume that $x$ and $y$ are linearly independent. Let us set
$$\varphi_{x, y}(t): = \frac{\|Tx + tTy\|}{\|x + ty\|} \qquad (t\in \mathbb{R}).$$
We have
$$(\varphi_{x, y})_{{\pm}}'(t): = \frac{\rho_{\pm}(Tx + tTy, Ty)\|x + ty\| - \rho_{\pm}(x + ty, y)\|Tx + tTy\|}{\|x + ty\|^2} \qquad (t\in \mathbb{R}).$$
Simple computations show that
$$\rho_*\left(x + ty, \frac{-\rho_{\pm}(x + ty, y)}{\|x + ty\|^2}(x + ty) + y\right) = 0 \qquad (t\in \mathbb{R}).$$
Hence our assumption yields that
$$Tx + tTy\perp_B \frac{-\rho_{\pm}(x + ty, y)}{\|x + ty\|^2}(Tx + tTy) + Ty.$$
It follows from \cite[Proposition 2.1.7]{A.S.T} that
\begin{align*}
\rho_-&\left(Tx + tTy, \frac{-\rho_{\pm}(x + ty, y)}{\|x + ty\|^2}(Tx + tTy) + Ty\right)
\\&\leq 0
\\& \leq \rho_+\left(Tx + tTy, \frac{-\rho_{\pm}(x + ty, y)}{\|x + ty\|^2}(Tx + tTy) + Ty\right) \qquad (t\in \mathbb{R}).
\end{align*}
This implies
$$\frac{-\rho_-(x + ty, y)}{\|x + ty\|^2}\|Tx + tTy\|^2 + \rho_-(Tx + tTy, Ty)\leq 0 \qquad (t\in \mathbb{R})$$
and
$$0 \leq \frac{-\rho_+(x + ty, y)}{\|x + ty\|^2}\|Tx + tTy\|^2 + \rho_+(Tx + tTy, Ty) \qquad (t\in \mathbb{R}).$$
We conclude that
$$0 \leq (\varphi_{x, y})_{-}'(t) \quad \mbox{and}\quad (\varphi_{x, y})_{+}'(t) \leq 0 \qquad (t\in \mathbb{R}).$$
Hence $\varphi_{x, y}$ is constant on $\mathbb{R}$. Therefore,
$$\frac{\|Tx\|}{\|x\|} = \varphi_{x, y}(0) = \lim_{t\rightarrow \infty}\varphi_{x, y}(t) = \frac{\|Ty\|}{\|y\|}.$$
Consequently, (iii) is valid.

Now, we show (iii)$\Rightarrow$(v). If (iii) holds, then the mapping $U=\frac{T}{\|T\|}:X\longrightarrow X$ is an isometry on $X$. By \cite[Theorem 1]{K.R}, there exists a semi-inner product $[\cdot|\cdot]\,:X\times X\longrightarrow \mathbb{R}$ such that $[Ux|Uy] = [x|y]$ for all $x, y\in X$. Thus $[Tx|Ty] = \|T\|^2[x|y]$ for all $x, y\in X$.
\end{proof}
Recall that a normed space $(X, \|\cdot\|)$ is called uniformly smooth if $X$ satisfies the property that for every $\varepsilon>0$ there exists $\delta>0$ such that if $x,y\in X$ with $\|x\|=1$ and $\|y\|\leq\delta$, then $\|x+y\|+\|x-y\| \leq 2 + \varepsilon\|y\|$; cf. \cite{A.S.T}.

The modulus of smoothness of $X$ is the function $\varrho_X$ defined for every $t > 0$ by the formula
$$\varrho_X(t) = \sup \left\{ \frac{\|x + y \| + \|x - y\|}{2} - 1 \,:\, \|x\| = 1, \; \|y\| = t \right\}.$$
Furthermore, $X$ is called uniformly convex if for every $0<\varepsilon \leq 2$ there is some $\delta>0$ such that for any two vectors with $\|x\| = \|y\| = 1$, the condition $\|x-y\|\geq\varepsilon$ implies that $\left\|\frac{x+y}{2}\right\|\leq 1-\delta.$

The modulus of convexity of $X$ is the function $\sigma_X$ defined by
$$\sigma_X (\varepsilon) = \inf \left\{ 1 - \left\| \frac{x + y}{2} \right\| \,:\, \|x\| = \|y\| = 1, \| x - y \| \geq \varepsilon \right\}.$$

Let $X, Y$ be normed spaces. If a linear mapping $T\,:X\longrightarrow Y$ preserves the $\rho_*$-orthogonality, then from Theorem \ref{L1} we conclude that $T$ must be a similarity. Thus, the spaces $X$ and $Y$ have to share some
geometrical properties. In particular, the modulus of convexity $\delta_X$ and modulus
of smoothness $\varrho_X$ must be preserved, i.e., $\sigma_X = \sigma_{T(X)}$ and $\varrho_X = \varrho _{T(X)}$. As a
consequence, we have the following result.
\begin{corollary}\label{cr.7.9}
Let $X$ be a normed space. Suppose that there exists a normed space $Y$ which is a uniformly convex (uniformly smooth) space, a strictly convex space, or an inner product space and a nontrivial linear mapping $T$ from $X$ into $Y$ (or from $Y$ onto $X$) such that $T$ preserves the $\rho_*$-orthogonality. Then $X$ is, respectively, a uniformly convex (uniformly smooth) space, a strictly convex space, an inner product space.
\end{corollary}

Recall that a normed space $(X, \|\cdot\|)$ is equivalent to an inner product space if there exist an inner product in $X$
and a norm $|||\cdot|||$ generated by this inner product such that
$$\frac{1}{k}\|x\| \leq |||x||| \leq k \|x\| \qquad (x\in X)$$
holds for some $k\geq1$; see \cite{Jo}.

\begin{corollary}\label{cr.8}
Any one of the following assertions implies that $X$ is equivalent to an inner product space.
\begin{itemize}
\item[(i)] There exist a normed space $Y$ satisfying the $\delta$-parallelogram law for some $\delta\in[0, 1)$ and a nonzero linear mapping $T: X\longrightarrow Y$ such that $T$ preserves the $\rho_*$-orthogonality.
\item[(ii)] There exist a normed space $Y$ satisfying the $\delta$-parallelogram law for some $\delta\in[0, 1)$ and a nonzero surjective linear mapping $S: Y\longrightarrow X$ such that $S$ preserves the $\rho_*$-orthogonality.
\end{itemize}
\end{corollary}
\begin{proof}
Suppose (i) holds. By Theorem \ref{L1}, there exists $\gamma>0$ such that $\|Tx\| = \gamma \|x\|$ for all $x\in X$. Hence
\begin{align}\label{id.71}
\|Tx + Ty\|^2 + \|Tx - Ty\|^2 = \gamma^2\left(\|x + y\|^2 + \|x - y\|^2\right) \qquad (x, y\in X).
\end{align}
Further, since the norm in $Y$ satisfies the $\delta$-parallelogram law (\ref{id.0011}), we get
\begin{align}\label{id.72}
\gamma^2\Big(2(1 - \delta)\|x\|^2 + 2\|y\|^2\Big) &= 2(1 - \delta)\|Tx\|^2 + 2\|Ty\|^2\nonumber
\\& \leq \|Tx + Ty\|^2 + \|Tx - Ty\|^2
\\& \leq 2(1 + \delta)\|Tx\|^2 + 2\|Ty\|^2 \nonumber
\\& = \gamma^2\Big(2(1 + \delta)\|x\|^2 + 2\|y\|^2\Big)\nonumber.
\end{align}
Therefore, by (\ref{id.71}) and (\ref{id.72}), we reach
$$2(1 - \delta)\|x\|^2 \leq \|x + y\|^2 + \|x - y\|^2 - 2\|y\|^2 \leq 2(1 + \delta)\|x\|^2 \qquad(x ,y \in X).$$
Hence the norm in $X$ satisfies the $\delta$-parallelogram law. As in the proof of Proposition \ref{pr.123456}, there exists an inner product in $X$
with the generated norm $|||.|||$ such that
\begin{align}\label{id.73}
\sqrt{1 - \delta}\|x\| \leq |||x||| \leq \sqrt{1 + \delta}\|x\| \qquad(x \in X).
\end{align}
It follows from $\sqrt{1 + \delta} \leq \frac{1}{\sqrt{1 - \delta}}$ and (\ref{id.73}) that
$$\sqrt{1 - \delta}\|x\| \leq |||x||| \leq \frac{1}{\sqrt{1 - \delta}}\|x\| \qquad(x \in X).$$
Thus $X$ is equivalent to an inner product space. By the same argument, if (ii) holds we deduce that $X$ is equivalent to an inner product space.
\end{proof}
We remark that the converse of Corollary \ref{cr.8} holds also true. Indeed, if $X$ is equivalent to an inner product space, then we can choose $\delta=0$, $Y=X$ and $T=id$, the identity operator on $X$.
\begin{theorem}
Let $X,Y$ be normed spaces. Then the following conditions are equivalent:
\begin{itemize}
\item[(i)] A linear mapping $T: X\longrightarrow Y$ preserves the $\rho_*$-orthogonality if and only if it is
a scalar multiple of a linear isometry.
\item[(ii)] A linear mapping $T: X\longrightarrow X$ preserves the $\rho_*$-orthogonality if and only if it is
a scalar multiple of a linear isometry.
\item[(iii)] Two $\rho_*$-orthogonality relations on $X$, generated by two norms on $X$, are equivalent if and only if these two norms are proportional.
\end{itemize}
\end{theorem}
\begin{proof}
The implications (i)$\Rightarrow$(ii) and (ii)$\Rightarrow$(iii) are trivial. To prove (iii)$\Rightarrow$(i), assume that (iii) holds. Let us consider an arbitrary linear mapping $T\,:X\longrightarrow Y$ that
preserves the $\rho_*$-orthogonality. Clearly, we can assume $T \neq 0$. Note that $T$ must be injective. Indeed, suppose otherwise that there
is a vector $x\in X$ with $\|x\| = 1$ such that $Tx = 0$. Take an arbitrary element $y\in X$ with $\|y\| = 1$. It is easily observed that $3x + y$ and $x$ are not Birkhoff--James orthogonal. Since $\perp_{\rho_+}\,\subseteq \,\perp_{B}$, we have $\rho_+(3x + y, x) \neq 0$.

Furthermore, a simple computation shows that $$\rho_*\left(3x + y, \frac{-\rho_+(3x + y, x)}{\|3x + y\|^2}(3x + y) + x\right) = 0.$$
As $T$ is $\rho_*$-orthogonality preserving, we get $$\rho_*\left(3Tx + Ty, \frac{-\rho_+(3x + y, x)}{\|3x + y\|^2}(3Tx + Ty) + Tx\right) = 0,$$ whence $\frac{{\rho_+}^2(3x + y, x)}{\|3x + y\|^4}\|Ty\|^4 = 0$. Thus $Ty = 0$ and so $T = 0$, which gives rise to a contradiction.
Now, we define a norm on $X$ by $|||x||| : = \|Tx\|\, \,(x\in X)$. We observe that $\perp_{\rho_*, \|\cdot\|}\,\subseteq \,\perp_{\rho_*, |||\cdot|||}$.
Utilizing (iii), we see that the norms $\|\cdot\|$ and $|||\cdot|||$ are proportional and this shows that $T$ is
a scalar multiple of a linear isometry.
\end{proof}
Next, we formulate one of our main results.
\begin{theorem}
Let $X,Y$ be normed spaces and $T:X\rightarrow Y$ be a nonzero linear mapping such that
\begin{align}\label{Eq0}
x\perp_I y\Rightarrow Tx\perp_{\rho_*}Ty \qquad(x,y\in X).
\end{align}
Then
\begin{align*}
\frac{1}{3}\|T\|\|x\|\leq\|Tx\|\leq 3[T]\|x\|\qquad(x\in X),
\end{align*}
where $[T]:=\inf\{\|Tx\|: \,x\in X, \|x\|=1\}$.
\end{theorem}
\begin{proof}
It is easy to see that \eqref{Eq0} is equivalent to
\begin{align}\label{Eq01}
\|x\|=\|y\|\Rightarrow T\Big(\frac{x+y}{2}\Big)\perp_{\rho_*}T\Big(\frac{x-y}{2}\Big) \qquad(x,y\in X).
\end{align}
Suppose that $x,y\in X$ with $\|x\|=\|y\|=1$. Then \eqref{Eq01} follows that
\begin{align*}
T\Big(\frac{x+y}{2}\Big)\perp_{\rho_*}T\Big(\frac{x-y}{2}\Big).
\end{align*}
Hence either $T\left(\frac{x+y}{2}\right)\perp_{\rho_-}T\left(\frac{x-y}{2}\right)$ or $T\left(\frac{x+y}{2}\right)\perp_{\rho_+}T\left(\frac{x-y}{2}\right)$.

Let us assume that $T\left(\frac{x+y}{2}\right)\perp_{\rho_-}T\left(\frac{x-y}{2}\right)$. We claim that $\|Ty\|<3\|Tx\|$, which ensures that $\|T\|\leq 3[T]$. In order to prove our claim, note that if either $\left\|T\left(\frac{x+y}{2}\right)\right\|=0$ or $\left\|T\left(\frac{x-y}{2}\right)\right\|=0$, then $\|Tx\|=\|Ty\|$ and so $\|Ty\|\leq3\|Tx\|$.

Otherwise, if $\left\|T(\frac{x+y}{2})\right\|\,\left\|T\left(\frac{x-y}{2}\right)\right\|\neq0$, then for some fixed $\gamma\in(0,1)$,
\begin{align*}
\rho_-\Big(T\Big(\frac{x+y}{2}\Big),T\Big(\frac{x-y}{2}\Big)\Big)=0<\gamma\Big\|T\Big(\frac{x+y}{2}\Big)\Big\|\,\Big\|T\Big(\frac{x-y}{2}\Big)\Big\|.
\end{align*}
By the definition of $\rho_-$,
\begin{align*}
\lim_{t\rightarrow0^{-}}\frac{\left\|T\left(\frac{x+y}{2}\right)+tT\left(\frac{x-y}{2}\right)\right\|-\left\|T\left(\frac{x+y}{2}\right)\right\|}{t}
<\gamma\Big\|T\Big(\frac{x+y}{2}\Big)\Big\|\,\Big\|T\Big(\frac{x-y}{2}\Big)\Big\|.
\end{align*}
It follows that there exists $\delta_1<0$ such that
\begin{align*}
\frac{\left\|T\left(\frac{x+y}{2}\right)+tT\left(\frac{x-y}{2}\right)\right\|-\left\|T\left(\frac{x+y}{2}\right)\right\|}{t}
<\gamma\Big\|T\Big(\frac{x+y}{2}\Big)\Big\|\,\Big\|T\Big(\frac{x-y}{2}\Big)\Big\|
\end{align*}
for all $t\in[\delta_1,0)$. Thus
\begin{align*}
\Big\|T\Big(\frac{x+y}{2}\Big)\Big\|<\Big\|T\Big(\frac{x+y}{2}\Big)+tT\Big(\frac{x-y}{2}\Big)\Big\|
+\gamma\Big\|T\Big(\frac{x+y}{2}\Big)\Big\|\,\Big\|T\Big(t\frac{x-y}{2}\Big)\Big\|
\end{align*}
for all $t\in[\delta_1,0)$.
Since $$0=\rho_-\Big(T\Big(\frac{x+y}{2}\Big),T\Big(\frac{x-y}{2}\Big)\Big)\leq\rho_+\Big(T\Big(\frac{x+y}{2}\Big),T\Big(\frac{x-y}{2}\Big)\Big),$$
we have
\begin{align*}
-\gamma\Big\|T\Big(\frac{x+y}{2}\Big)\Big\|\,\Big\|T\Big(\frac{x-y}{2}\Big)\Big\|
<\rho_+\Big(T\Big(\frac{x+y}{2}\Big),T\Big(\frac{x-y}{2}\Big)\Big).
\end{align*}
By the definition of $\rho_+$,
\begin{align*}
-\gamma\Big\|T\Big(\frac{x+y}{2}\Big)\Big\|\,\Big\|T\Big(\frac{x-y}{2}\Big)\Big\|
<\lim_{t\rightarrow0^+}\frac{\left\|T\left(\frac{x+y}{2}\right)+tT\left(\frac{x-y}{2}\right)\right\|-\left\|T\left(\frac{x+y}{2}\right)\right\|}{t}.
\end{align*}
Consequently, there exists $\delta_2>0$ such that
\begin{align*}
\Big\|T\Big(\frac{x+y}{2}\Big)\Big\|<\Big\|T\Big(\frac{x+y}{2}\Big)+tT\Big(\frac{x-y}{2}\Big)\Big\|
+\gamma\Big\|T\Big(\frac{x+y}{2}\Big)\Big\|\,\Big\|T\Big(t\frac{x-y}{2}\Big)\Big\|
\end{align*}
for all $t\in(0,\delta_2]$. Let us define $\varphi:\mathbb{R}\rightarrow\mathbb{R}$ by
\begin{align*}
\varphi(t):=\Big\|T\Big(\frac{x+y}{2}\Big)+tT\Big(\frac{x-y}{2}\Big)\Big\|
+\gamma\Big\|T\Big(\frac{x+y}{2}\Big)\Big\|\,\Big\|T\Big(t\frac{x-y}{2}\Big)\Big\|.
\end{align*}
Then $\varphi$ is convex and $\varphi(t)>\varphi(0)$ for all $t\in[\delta_1,0)\cup(0,\delta_2]$, which yields that $\varphi(t)>\varphi(0)$ for all $t\in\mathbb{R}\setminus\{0\}$. Therefore,
\begin{align*}
\Big\|T\Big(\frac{x+y}{2}\Big)\Big\|<\Big\|T\Big(\frac{x+y}{2}\Big)+T\Big(\frac{x-y}{2}\Big)\Big\|
+\gamma\Big\|T\Big(\frac{x+y}{2}\Big)\Big\|\,\Big\|T\Big(\frac{x+y}{2}\Big)\Big\|.
\end{align*}
Finally if $\gamma\rightarrow0^+$, then
\begin{align*}
\Big\|T\Big(\frac{x+y}{2}\Big)\Big\|<\|Tx\|.
\end{align*}
It follows that $\|Tx+Ty\|<2\|Tx\|$. Therefore, $\|Ty\|-\|Tx\|\leq\|Tx+Ty\|<2\|Tx\|$ from which we get $\|Ty\|<3\|Tx\|$.
In the case when $T(\frac{x+y}{2})\perp_{\rho_+}T(\frac{x-y}{2})$, our claim can be established by a similar argument.
\end{proof}

\begin{corollary}\label{pr.1.2}
Let $X$ be a normed space endowed with two norms ${\|\cdot\|}_1$ and ${\|\cdot\|}_2$. Then the following conditions are equivalent:
\begin{itemize}
\item[(i)] $\frac{\rho_{*, 1}(x, y)}{\|x\|^4_1} = \frac{\rho_{*, 2}(x, y)}{\|x\|^4_2} \quad (x, y\in X)$.
\item[(ii)] There exist $m, M>0$ such that $$m|\rho_{*, 1}(x, y)| \leq |\rho_{*, 2}(x, y)| \leq M |\rho_{*, 1}(x, y)| \qquad (x, y\in X).$$
\item[(iii)] There exists $M>0$ such that $|\rho_{*, 2}(x, y)| \leq M |\rho_{*, 1}(x, y)| \quad (x, y\in X)$.
\item[(iv)] There exists $m>0$ such that $m|\rho_{*, 1}(x, y)| \leq |\rho_{*, 2}(x, y)| \quad (x, y\in X)$.
\item[(v)] There exist $m, M\in\mathbb{R}$ such that $$m\rho_{*, 1}(x, y) \leq \rho_{*, 2}(x, y) \leq M \rho_{*, 1}(x, y) \qquad (x, y\in X).$$
\item[(vi)] There exists $M>0$ such that $|\rho_{*, 2}(x, y)| = M |\rho_{*, 1}(x, y)| \quad (x, y\in X)$.
\item[(vii)] There exists $M>0$ such that $\rho_{*, 2}(x, y) = M \rho_{*, 1}(x, y) \quad (x, y\in X)$.
\item[(viii)] There exists $M>0$ such that ${\|x\|}_2 = M {\|x\|}_1 \quad (x\in X)$.
\end{itemize}
\end{corollary}
\begin{proof}
(i)$\Rightarrow$(viii): Assume that $T=id:(X, {\|\cdot\|}_1)\rightarrow(X, {\|\cdot\|}_2)$ is the identity map. Let $x,y\in X$ and $\rho_{*,1}(x,y)=0$. Then (i) implies $\rho_{*,2}(Tx,Ty)=\rho_{*,2}(x,y)=0$. Therefore, Theorem \ref{L1} follows that there exists $M>0$ such that ${\|x\|}_2={\|Tx\|}_2=M{\|x\|}_1$. The other implications can be proved similarly.
\end{proof}
Let us adopt the notion of Birkhoff orthogonal set of $x$ from \cite{A.S.T}:
$$[x]^B_{\|\cdot\|} = \{y\in X \,: \,\,x\perp_{B}y\}.$$
We now define the $\diamondsuit$-orthogonal set of $x$ as follows:
$$[x]^{\diamondsuit}_{\|\cdot\|} = \{y\in X \,: \,\,x\perp_{\diamondsuit}y\},$$
where $\diamondsuit \in \{I, \rho_{+}, \rho_{-}, \rho, \rho_* \}$.
\begin{corollary}
Let $X$ be a normed space endowed with two norms ${\|\cdot\|}_1$ and ${\|\cdot\|}_2$. For every $x\in X$, the following conditions are equivalent:
\begin{align*}
&{\rm(i)}\,[x]^{\rho_*}_{{\|\cdot\|}_1} = [x]^{\rho_*}_{{\|\cdot\|}_2}.\qquad \quad{\rm(ii)}\,[x]^{B}_{{\|\cdot\|}_1} = [x]^{B}_{{\|\cdot\|}_2}.\\
&{\rm (iii)}\,[x]^{I}_{{\|\cdot\|}_1} = [x]^{I}_{{\|\cdot\|}_2}.\qquad \,\,\,{\rm (iv)}\,[x]^{\rho}_{{\|\cdot\|}_1} = [x]^{\rho}_{{\|\cdot\|}_2}.\\
&{\rm (v)}\,[x]^{\rho_{+}}_{{\|\cdot\|}_1} = [x]^{\rho_{+}}_{{\|\cdot\|}_2}.\qquad \,\,\,\,\,{\rm (vi)}\,[x]^{\rho_{-}}_{{\|\cdot\|}_1} = [x]^{\rho_{-}}_{{\|\cdot\|}_2}.
\end{align*}
\end{corollary}
\begin{proof}
(i)$\Rightarrow$(ii): Suppose that (i) holds and define $T=id:(X,{\|\cdot\|}_1)\rightarrow(X,{\|\cdot\|}_2)$ to be the identity map. Then $T$ is $\rho_*$-orthogonal preserving. It follows from Theorem \ref{L1} that there exists $M>0$ such that ${\|Tx\|}_2={\|x\|}_2=M{\|x\|}_1$, which implies that $[x]^{B}_{{\|\cdot\|}_1} = [x]^{B}_{{\|\cdot\|}_2}\, (x\in X)$.

(ii)$\Rightarrow$(i): If (ii) holds, then $T=id:(X,{\|\cdot\|}_1)\rightarrow(X,{\|\cdot\|}_2)$ is $B$-orthogonal preserving. It follows from \cite[Theorem 3.1]{B.T} that there exists $M>0$ such that ${\|x\|}_2={\|Tx\|}_2=M{\|x\|}_1$, which ensures that $[x]^{\rho_*}_{{\|\cdot\|}_1} = [x]^{\rho_*}_{{\|\cdot\|}_2}
\, (x\in X)$. The other implications can be proved similarly.
\end{proof}


\section{Characterizations of smooth spaces}
The relations $\perp_{\rho}$ and $\perp_{\rho_*}$ are generally incomparable. Then the following results give some characterizations of the smooth normed spaces.
\begin{theorem}\label{Th1}
Let $(X,\|\cdot\|)$ be a normed space. The following conditions are equivalent:
\begin{align*}
&{\rm (i)}\perp_{B}\subseteq\perp_{\rho_*}.\qquad {\rm (ii)}\perp_{B} = \perp_{\rho_*}.\qquad {\rm (iii)} \perp_{\rho}\subseteq\perp_{\rho_*}.\\
&{\rm (iv)}\perp_{\rho_*}\subseteq\perp_{\rho}.\quad \,\,\,\,\,{\rm (v)} \perp_{\rho_*} = \perp_{\rho}.\qquad {\rm (vi)} \perp_{\rho_*}\subseteq\perp_{\rho_+}.\\
&{\rm (vii)} \perp_{\rho_*}\subseteq\perp_{\rho_-}.\,\,\,\,\,\, {\rm (viii)} \perp_{\rho_*} = \perp_{\rho_-}.\,\,\,\,\, {\rm (ix)}\mbox{ $X$ is smooth}.
\end{align*}
\end{theorem}
\begin{proof}
It is well known that $X$ is smooth if and only if $\rho_-(x,y)=\rho_+(x,y)$ for every $x,y\in X$; see e.g., \cite[Remark 2.1.1]{A.S.T}.

First, we prove (i)$\Leftrightarrow$(ix). Suppose that $X$ is smooth and $x,y\in X$ such that $x\perp_{B}y$. Then \cite[Proposition 2.2.2]{A.S.T} ensures that $x\perp_{\rho_+}y$ and this yields that $\rho_*(x,y)=0$. Now, assume that (i) holds and $x,y\in X$ with $x\neq0$. It is clear that
\begin{align*}
\rho_-(x,y)\leq\frac{\alpha\rho_-(x,y)+(1-\alpha)\rho_+(x,y)}{\|x\|^2}\|x\|^2\leq\rho_+(x,y)
\end{align*}
for every $\alpha\in[0,1]$. It follows from \cite[Proposition 2.1.7]{A.S.T} that $x\perp_{B}tx+y$ with $t=\frac{-\alpha\rho_-(x,y)-(1-\alpha)\rho_+(x,y)}{\|x\|^2}$. By the assumption, we get $x\perp_{\rho_*}tx+y$. Hence $\rho_*(x,x+ty)=0$. On the other hand,
\begin{align*}
\rho_*(x,tx+y)&=t^2\|x\|^4+t\|x\|^2\Big(\rho_-(x,y)+\rho_+(x,y)\Big)+\rho_*(x,y)\\
&=\Big(-\alpha\rho_-(x,y)-(1-\alpha)\rho_+(x,y)\Big)^2\\
&\quad+\Big(-\alpha\rho_-(x,y)-(1-\alpha)\rho_+(x,y)\Big)\Big(\rho_-(x,y)+\rho_+(x,y)\Big)\\
&\quad+\rho_-(x,y)\rho_+(x,y)\\
&=\alpha(\alpha-1)\Big(\rho_-(x,y)-\rho_+(x,y)\Big)^2.
\end{align*}
Therefore, $\rho_-(x,y)=\rho_+(x,y)$ ensures that $X$ is smooth.

Now, we prove (iii)$\Leftrightarrow$(ix). Let $x,y\in X$. Clearly $x\perp_{\rho}\left(-\frac{\rho(x,y)}{\|x\|^2}x+y\right)$. It follows from (iii) that $x\perp_{\rho_*}\left(-\frac{\rho(x,y)}{\|x\|^2}x+y\right)$. Hence
\begin{align*}
\rho_*\left(x,-\frac{\rho(x,y)}{\|x\|^2}x+y\right)&=\frac{\rho^2(x,y)}{\|x\|^4}\|x\|^4-2\frac{\rho^2(x,y)}{\|x\|^2}\|x\|^2+\rho_*(x,y)\\
&=\rho_*(x,y)-\rho^2(x,y)=0.
\end{align*}
Thus $\rho^2(x,y)=\rho_*(x,y)=\rho_-(x,y)\rho_+(x,y)$. Hence
\begin{align*}
\Big(\rho_-(x,y)+\rho_+(x,y)\Big)^2=4\rho_-(x,y)\rho_+(x,y).
\end{align*}
Therefore, $\Big(\rho_-(x,y)-\rho_+(x,y)\Big)^2=0$. Finally, we get $\rho_-(x,y)=\rho_+(x,y)$. It ensures that $X$ is smooth.

To show (iv)$\Leftrightarrow$(ix), let $x,y\in X$. From $x\perp_{\rho_+}\left(-\frac{\rho_+(x,y)}{\|x\|^2}x+y\right)$ we deduce that $x\perp_{\rho_*}\left(-\frac{\rho_+(x,y)}{\|x\|^2}x+y\right)$. It follows from (iv) that $x\perp_{\rho}\left(-\frac{\rho_+(x,y)}{\|x\|^2}x+y\right)$. Hence
\begin{align*}
\rho\left(x,-\frac{\rho_+(x,y)}{\|x\|^2}x+y\right)=-\frac{\rho_+(x,y)}{\|x\|^2}\|x\|^2+\rho(x,y)=\frac{\rho_-(x,y)-\rho_+(x,y)}{2}=0.
\end{align*}
Therefore, $\rho_-(x,y)=\rho_+(x,y)$, which follows that $X$ is smooth.

Finally, we prove (vi)$\Leftrightarrow$(ix). Let $x,y\in X$. Then
\begin{align*}
\rho_*\left(x,-\frac{\rho_-(x,y)}{\|x\|^2}x+y\right)&=\rho_-\left(x,-\frac{\rho_-(x,y)}{\|x\|^2}x+y\right)\rho_+\left(x,-\frac{\rho_-(x,y)}{\|x\|^2}x+y\right)\\
&=\left(-\frac{\rho_-(x,y)}{\|x\|^2}\|x\|^2+\rho_-(x,y)\right)\rho_+\left(x,-\frac{\rho_-(x,y)}{\|x\|^2}x+y\right)\\
&=\Big(\rho_-(x,y)-\rho_-(x,y)\Big)\rho_+\left(x,-\frac{\rho_-(x,y)}{\|x\|^2}x+y\right)=0.
\end{align*}
Hence $x\perp_{\rho_*}\left(-\frac{\rho_-(x,y)}{\|x\|^2}x+y\right)$. It follows from (vi) that $x\perp_{\rho_+}\left(-\frac{\rho_-(x,y)}{\|x\|^2}x+y\right)$. Therefore,
\begin{align*}
\rho_+\left(x,-\frac{\rho_-(x,y)}{\|x\|^2}x+y\right)=-\frac{\rho_-(x,y)}{\|x\|^2}\|x\|^2+\rho_+(x,y)=0,
\end{align*}
which yields that $\rho_-(x,y)=\rho_+(x,y)$. Thus $X$ is smooth.

The other implications can be shown similarly.
\end{proof}
In the following, we give some characterizations of smooth spaces in terms of $\rho_*$-orthogonal preserving mappings.
\begin{theorem}
Let $X$ be a normed space. The following assertions are equivalent:
\begin{itemize}
\item[(i)] $X$ is smooth.
\item[(ii)] There exist a normed space Y and a nonvanishing linear mapping
$T\,:X\longrightarrow Y$ such that
$x\perp_{\rho_*}y \Longrightarrow Tx\perp_{\rho_+}Ty \quad (x, y \in X)$.
\item[(iii)] There exist a normed space Y and a nonvanishing linear mapping
$T\,:X\longrightarrow Y$ such that
$x\perp_{\rho_*}y \Longrightarrow Tx\perp_{\rho_-}Ty \quad (x, y \in X).$
\end{itemize}
\end{theorem}
\begin{proof}
To prove (i)$\Rightarrow$(ii), let $X$ be smooth. By Theorem \ref{Th1}, $\perp_{\rho_*}=\perp_{\rho_+}$. Put $Y=X$ and $T=id$.

Now, we prove (ii)$\Rightarrow$(iii). By (ii), there exists a normed space Y and a nonvanishing linear mapping $T\,:X\longrightarrow Y$ such that
$x\perp_{\rho_*}y$ implies $Tx\perp_{\rho_+}Ty$ for all $x, y\in X$. Let $x,y\in X$ such that $x\perp_{\rho_*}y$. Then $\rho_*(x,-y)=\rho_*(x,y)=0$. It follows that $x\perp_{\rho_*}-y$. Hence $\rho_-(Tx,Ty)=-\rho_+(Tx,-Ty)=0$. Therefore, $Tx\perp_{\rho_+}Ty$.

To establish (iii)$\Rightarrow$(i), let $x,y\in X$ and $x\neq0$. Since $x\perp_{\rho_*}\frac{-\rho_-(x,y)}{\|x\|^2}x+y$, we have
\begin{align*}
\rho_+\left(Tx,\frac{-\rho_-(x,y)}{\|x\|^2}Tx+Ty\right)=\frac{-\rho_-(x,y)}{\|x\|^2}\|Tx\|^2+\rho_+(Tx,Ty)=0.
\end{align*}
Therefore,
\begin{align}\label{Eq1}
\rho_+(Tx,Ty)=\frac{\|Tx\|^2}{\|x\|^2}\rho_-(x,y).
\end{align}
In addition, we have $x\perp_{\rho_*}\frac{\rho_+(x,y)}{\|x\|^2}x-y$. It follows from (iii) that
\begin{align*}
\rho_+\left(Tx,\frac{\rho_+(x,y)}{\|x\|^2}Tx-Ty\right)=\frac{\rho_+(x,y)}{\|x\|^2}\|Tx\|^2-\rho_-(Tx,Ty)=0,
\end{align*}
whence
\begin{align}\label{Eq2}
\rho_-(Tx,Ty)=\frac{\|Tx\|^2}{\|x\|^2}\rho_+(x,y).
\end{align}
 It follows from \eqref{Eq1} and \eqref{Eq2} that
 \begin{align*}
 \rho_*(Tx,Ty)=\frac{\|Tx\|^4}{\|x\|^4}\rho_*(x,y).
 \end{align*}
 Hence $T$ is $\rho_*$-orthogonal preserving mapping. An application of Theorem \ref{L1} yields that $\|Tx\|=\|T\|\,\|x\|$ for all $x\in X$. Consequently, using \cite[Theorem 4.3.1]{W} and \eqref{Eq1}, we have
 \begin{align*}
 \rho_+(x,y)=\frac{\rho_+(Tx,Ty)}{\|T\|^2}=\frac{\|T\|^2\rho_-(x,y)}{\|T\|^2}=\rho_-(x,y).
 \end{align*}
 Thus $X$ is smooth.
\end{proof}
\begin{theorem}
Let $X$ be a normed space. Then the following conditions are equivalent:
\begin{itemize}
\item[(i)] $X$ is smooth.
\item[(ii)] There exist a normed space Y and a nonvanishing linear mapping
$T: X\longrightarrow Y$ such that
$x\perp_{\rho_*}y \Longrightarrow Tx\perp_{\rho}Ty \quad (x, y \in X)$.
\item[(iii)] There exist a normed space Y and a nonvanishing linear mapping
$T: X\longrightarrow Y$ such that
$x\perp_{\rho}y \Longrightarrow Tx\perp_{\rho_*}Ty \quad (x, y \in X).$
\end{itemize}
\end{theorem}
\begin{proof}
First, we prove (i)$\Rightarrow$(ii) and (i)$\Rightarrow$(iii). If $X$ is smooth, then Theorem \ref{Th1} shows that $\perp_{\rho}=\perp_{\rho_*}$. So it is enough to put $Y=X$ and $T=id$.\\
To prove (iii)$\Rightarrow$(i), assume that (iii) holds and $x,y\in X$ with $x\neq0$. We know that $x\perp_{\rho}\frac{-\rho(x,y)}{\|x\|^2}x+y$. It follows from (iii) that $\rho_*\big(Tx,\frac{-\rho(x,y)}{\|x\|^2}Tx+Ty\big)=0$, which implies that
$$ \rho_+\left(Tx,\frac{-\rho(x,y)}{\|x\|^2}Tx+Ty\right)=0\qquad {\rm or} \qquad \rho_-\left(Tx,\frac{-\rho(x,y)}{\|x\|^2}Tx+Ty\right)=0.$$
Hence
$$
\rho_+(Tx,Ty)=\frac{\|Tx\|^2}{\|x\|^2}\rho(x,y)\qquad {\rm or} \qquad \rho_-(Tx,Ty)=\frac{\|Tx\|^2}{\|x\|^2}\rho(x,y).$$
Therefore,
$$ x\perp_{\rho}y\Rightarrow Tx\perp_{\rho_-}Ty \qquad {\rm or} \qquad x\perp_{\rho}y\Rightarrow Tx\perp_{\rho_+}Ty.$$
Now, \cite[Theorem 5.1]{W} concludes that $X$ is smooth.

To show (ii)$\Rightarrow$(i), suppose that (ii) holds, $x,y\in X$ and $x\neq0$. Since $x\perp_{\rho_*}\frac{-\rho_+(x,y)}{\|x\|^2}x+y$, we have $Tx\perp_{\rho}\frac{-\rho_+(x,y)}{\|x\|^2}Tx+Ty$. It yields that
\begin{align}\label{Eq6}
\rho(Tx,Ty)=\frac{\|Tx\|^2}{\|x\|^2}\rho_+(x,y).
\end{align}
Moreover, $x\perp_{\rho_*}\frac{-\rho_-(x,y)}{\|x\|^2}x+y$. Similarly, we get
\begin{align}\label{Eq7}
\rho(Tx,Ty)=\frac{\|Tx\|^2}{\|x\|^2}\rho_-(x,y).
\end{align}
Now, from \eqref{Eq6} and \eqref{Eq7}, we conclude that $\rho_-(x,y)=\rho_+(x,y)$. Thus $X$ is smooth.
\end{proof}


\section{$\rho_*$-orthogonal additivity}
Let $X$ be a normed space. In this section, we will show that the pair $(X,\perp_{\rho_*})$ is an orthogonality space in the sense of R\"{a}tz. The conditions (O1)--(O3) follow easily from Proposition \ref{pr1}. In order to prove (O4), we need the following Lemma.
\begin{lemma}\label{lm1}
Let $(X,\|\cdot\|)$ be a normed space. Then for all $x\in X\setminus\{0\}$ and all $\lambda\geq0$ there exists $z\in X\setminus\{0\}$ such that
\begin{align*}
\rho_-(x,z)\rho_-(z,x)=\frac{\|x\|^2\|z\|^2}{1+\lambda}.
\end{align*}
\end{lemma}
\begin{proof}
If $\lambda=0$, then choose $z:=x$. Now, let $\lambda >0$. First, note that there exists $w_0\in X$, which is linearly independent of $x$ such that $\rho_+(x,w_0)\neq0$. Indeed, If $\rho_+(x,w)=0$ for each $w\in X$ being linearly independent of $x$, then $\rho_+(x,x+w_0)=0$. Hence $-\|x\|^2=\rho_+(x,w_0)=0$. Then $x=0$, which yields a contradiction. Moreover, there exists $w\in X$ which is linearly independent of $x$ such that $\rho_+(x,w)>0$. In fact, if $\rho_+(x,w)<0$ for any $w\in X$ being linearly independent of $x$, then $\rho_-(x,w_0)\leq\rho_+(x,w_0)<0$.

Further, since $x$ and $x+\rho_+(x,w_0)w_0$ are linearly independent, we conclude that $\rho_+\Big(x,x+\rho_+(x,w_0)w_0\Big)<0$. Due to
\begin{align*}
\rho_+\Big(x,x+\rho_+(x,w_0)w_0\Big)=\|x\|^2+\rho_+(x,w_0)\rho_-(x,w_0),
\end{align*}
we get $\|x\|^2+\rho_+(x,w_0)\rho_-(x,w_0)<0$, which is impossible because of
\begin{align*}
\|x\|^2+\rho_+(x,w_0)\rho_-(x,w_0)\geq\rho_+(x,w_0)\rho_-(x,w_0)\geq\rho_+(x,w_0)>0.
\end{align*}
Now, define
\begin{align*}
\varphi(t) :=\frac{\|x\|^2\|x+tw\|^2}{1+\lambda}-\rho_-(x,x+tw)\rho_-(x+tw,x).
\end{align*}
By \cite[Proposition 2.1.3]{A.S.T} and \cite[Lemma 2.8.2]{A.S.T}, $\varphi$ is continuous on $\mathbb{R}$. Moreover,
\begin{align*}
\varphi(0)=\|x\|^4\big(\frac{1}{1+\lambda}-1\big)<0.
\end{align*}
If $t_1=-\frac{\|x\|^2}{\rho_+(x,w)}<0$, then
\begin{align*}
\rho_-(x,x+t_1w)=\|x\|^2+t_1\rho_+(x,w) =\|x\|^2-\frac{\|x\|^2}{\rho_+(x,w)}\rho_+(x,w)=0.
\end{align*}
It ensures that
\begin{align*}
\varphi(t_1)=\frac{\|x\|^2\|x+t_1w\|^2}{1+\lambda}>0.
\end{align*}
Employing the mean value theorem we infer that there exists a number $t_0$ between $0$ and $t_1$ such that $\varphi(t_0)=0$. To complete the proof, it is enough to put $z=x+t_0w$.
\end{proof}
We are now in the position to establish the main result of this section.
\begin{theorem}\label{Lm2}
For any two-dimensional subspace $P$ of normed space $(X,\|\cdot\|)$ and for
every $x\in P$, $\lambda\geq0$, there exists a vector $y\in P$ such that
$$x\perp_{\rho_*} y\qquad {\rm and} \qquad x+y\perp_{\rho_*}\lambda x-y.$$
Conversely, the pair $(X,\perp_{\rho_*})$ is an orthogonality space in the sense of R\"{a}tz.
\end{theorem}
\begin{proof}
Fix $x\in X$. If $x=0$, take $z=0$. By Lemma \ref{lm1}, for $x\neq0$ take a nonzero element $z\in X$ such that
\begin{align}\label{Eq8}
\rho_-(x,z)\rho_-(z,x)=\frac{\|x\|^2\|z\|^2}{1+\lambda}.
\end{align}
Let
$y:=-x+\frac{1+\lambda}{\|z\|^2}\rho_-(z,x)z$. We consider two cases:

$\mathbf{Case \,1.}$ Suppose that $\rho_+(z,x)>0$. Then
\begin{align*}
\rho_-(x,y)&=\rho_+\left(x,-x+\frac{1+\lambda}{\|z\|^2}\rho_-(z,x)z\right)\\
&=-\|x\|^2+\frac{1+\lambda}{\|z\|^2}\rho_-(z,x)\rho_-(x,z)\\
&=-\|x\|^2+\frac{1+\lambda}{\|z\|^2}\times\frac{\|x\|^2\|z\|^2}{1+\lambda}\qquad\mbox{(by \eqref{Eq8})}\\
&=-\|x\|^2+\|x\|^2=0.
\end{align*}
Thus $x\perp_{\rho_-}y$. Therefore, $x\perp_{\rho_*}y$. Now, assume that $\rho_-(z,x)<0$. Then
\begin{align*}
\rho_+(x,y)&=\rho_+\left(x,-x+\frac{1+\lambda}{\|z\|^2}\rho_-(z,x)z\right)\\
&=-\|x\|^2+\frac{1+\lambda}{\|z\|^2}\rho_-(z,x)\rho_-(x,z)\\
&=-\|x\|^2+\frac{1+\lambda}{\|z\|^2}\times\frac{\|x\|^2\|z\|^2}{1+\lambda} \qquad\mbox{(by \eqref{Eq8})}\\
&=-\|x\|^2+\|x\|^2=0.
\end{align*}
Hence $x\perp_{\rho_+}y$, whence $x\perp_{\rho_*}y$. Furthermore, if $\rho_-(z,x)>0$, then
\begin{align*}
\rho_+(x+y,y-\lambda x)&=\rho_+\left(\frac{1+\lambda}{\|z\|^2}\rho_-(z,x)z,\frac{1+\lambda}{\|z\|^2}\rho_-(z,x)z-(1+\lambda)x\right)\\
&=\frac{(1+\lambda)^2\rho^2_-(z,x)}{\|z\|^4}\|z\|^2+\rho_+\left(\frac{1+\lambda}{\|z\|^2}\rho_-(z,x)z,(1+\lambda)x\right)\\
&=\frac{(1+\lambda)^2\rho^2_-(z,x)}{\|z\|^2}-\frac{(1+\lambda)^2\rho^2_-(z,x)}{\|z\|^2}=0.
\end{align*}
It follows that $x+y\perp_{\rho_+}y-\lambda x$. Hence $x+y\perp_{\rho_-}\lambda x-y$ and so $x+y\perp_{\rho_*}\lambda x-y$.

$\mathbf{Case \,2.}$ Suppose that $\rho_+(z,x)<0$. We have
\begin{align*}
\rho_-(x+y,y-\lambda x)&=\rho_-\left(\frac{1+\lambda}{\|z\|^2}\rho_-(z,x)z,\frac{1+\lambda}{\|z\|^2}\rho_-(z,x)z-(1+\lambda)x\right)\\
&=\frac{(1+\lambda)^2\rho^2_-(z,x)}{\|z\|^4}\|z\|^2-\frac{(1+\lambda)^2\rho^2_-(z,x)}{\|z\|^2}=0,
\end{align*}
which yields that $x+y\perp_{\rho_-}y-\lambda x$. Then $x+y\perp_{\rho_+}\lambda x-y$. Therefore $x+y\perp_{\rho_*}\lambda x-y$.
\end{proof}
Let $X$ be a normed space and let $(G, +)$ be an Abelian group. Let us recall that a mapping $A\,:X\longrightarrow G$ is called additive if $A(x + y) = A(x)+A(y)$ for all $x, y \in X$, a mapping $B\,:X \times X\longrightarrow G$ is called biadditive if it is additive in both variables and a mapping $Q\,:X\longrightarrow G$ is called quadratic if $Q(x + y) + Q(x -y) = 2Q(x)+2Q(y)$ for all $x, y \in X$. As an immediate consequence of our results, we deduce the following assertion.
\begin{corollary}
Let $X$ be a normed space and let $(G, +)$ be an Abelian group. A mapping $f:X\longrightarrow G$ satisfies the condition
$$x\perp_{\rho_*}y \Longrightarrow f(x + y) = f(x) + f(y) \qquad (x, y \in X)$$
if and only if there exist an additive mapping $A\,:X\longrightarrow G$ and a biadditive and symmetric mapping $B\,:X \times X\longrightarrow G$ such that
$$f(x) = A(x) + B(x, x) \qquad(x\in X)$$
and
$$x\perp_{\rho_*}y \Longrightarrow B(x, y) = 0 \qquad (x, y \in X).$$
\end{corollary}
\begin{proof}
According to Theorem \ref{Lm2}, $(X,\perp_{\rho_*})$ is an orthogonality space. So, employing \cite[Theorem 2.8.1]{A.S.T} completes the proof.
\end{proof}
Finally, as a consequence of Theorem \ref{Lm2} and \cite[Theorem 3]{Mo}, we have the following result.
\begin{corollary}
Let $X$ be a normed space and let $(G, +)$ be an Abelian group. Suppose that $Y$ is a real Banach space. If $f:X\longrightarrow G$ is a mapping fulfilling
$$x\perp_{\rho_*}y \Longrightarrow \|f(x + y) - f(x) - f(y)\| \leq \varepsilon \qquad (x, y \in X)$$
for some $\varepsilon$, then there
exist exactly an additive mapping $A\,:X\longrightarrow Y$ and exactly a quadratic mapping
$Q\,:X\longrightarrow Y$ such that
$$\|f(x) - f(0) - A(x) - Q(x)\| \leq \frac{68}{3}\varepsilon, \qquad (x\in X).$$
\end{corollary}


 \end{document}